\documentclass[11pt]{amsart}
\usepackage{amsmath, amssymb, amsthm, bm, mathtools, enumitem, caption}
\usepackage{hyperref}

\usepackage[noabbrev,capitalize]{cleveref}
\usepackage{adjustbox}
\crefname{equation}{}{}
\usepackage{ytableau}
\usepackage{graphicx, tikz}
\usepackage[textheight=8.7in, textwidth=6.5in]{geometry}
\usepackage{lscape}

\usepackage{microtype}

\newtheorem{theorem}{Theorem}[section]
\newtheorem{lemma}[theorem]{Lemma}

\newtheorem{proposition}[theorem]{Proposition}
\newtheorem{conjecture}[theorem]{Conjecture}
\newtheorem*{conjecture*}{Conjecture}

\theoremstyle{definition}

\theoremstyle{remark}
\newtheorem*{remark}{Remark}
\newtheorem*{remarks}{Remarks}
\newtheorem*{example}{Example}
\newtheorem*{examples}{Examples}

\numberwithin{equation}{section}

\newcommand{\N}{\mathbb N}

\DeclareMathOperator{\Sym}{sym}

\def\lp{\left(}
\def\rp{\right)}

\newcommand{\Q}{\mathbb Q}

\newcommand{\Z}{\mathbb Z}

\title[Integer partitions detect the primes]{Integer partitions detect the primes}
\date{\today}

\author{William Craig, Jan-Willem van Ittersum \and Ken Ono}
\address{Department of Mathematics and Computer Science, University of Cologne, Weyertal 86-90, 50931 Cologne, Germany}
\email{wcraig@uni-koeln.de}
\email{j.w.ittersum@uni-koeln.de}
\address{Dept. of Mathematics, University of Virginia, Charlottesville, VA 22904}
\email{ko5wk@virginia.edu}
\date{}

\begin{document}
	\begin{abstract}  We show that integer partitions, the fundamental building blocks in additive number theory, detect prime numbers in an unexpected way. Answering a question of Schneider, we show that the primes are the solutions to special equations in partition functions.  For example, an integer $n\geq 2$ is prime if and only if
		$$
		(3n^3 - 13n^2 + 18n - 8)M_1(n) + (12n^2 - 120n + 212)M_2(n) -960M_3(n) = 0,
		$$
		where the $M_a(n)$ are MacMahon's well-studied partition functions. More generally, for {\it MacMahonesque} partition functions $M_{\vec{a}}(n),$ we prove that there are infinitely many such prime detecting equations with constant coefficients, such as 		\begin{displaymath}
		80M_{(1,1,1)}(n)-12M_{(2,0,1)}(n)+12M_{(2,1,0)}(n)+\dots-12M_{(1,3)}(n)-39M_{(3,1)}(n)=0.
		\end{displaymath}
			\end{abstract}
	
	\maketitle
	\section{Introduction and Statement of Results}
		
	{\it Diophantine sets} $S$ are subsets of $\Z_+^n$ that arise from the positive integer solutions of a Diophantine equation. For example, the set of composite numbers is Diophantine, which is easily seen by considering the polynomial equation
	$$
	C(x_1, x_2, x_3):= x_1-(x_2+1)(x_3+1)=0.
	$$
	Indeed, we have that $S_C:=\{ x_1 \in \Z_+ \ : \ C(x_1, x_2, x_3)=0\ {\text {\rm with}}\ x_2, x_3\in \Z_+\}$ is the set of composite numbers. In 1970, Matiyasevich \cite{Matiyasevich} famously proved that a set is Diophantine if and only if it is computably enumerable. Building on work by Davis, Putnam, and Robinson, this result resolved Hilbert's tenth problem in the negative: There is no algorithm that determines whether a generic Diophantine equation has integer solutions.
	
	In contrast to the simple case of the composite numbers, Matiyasevich's work implies that the set of primes is Diophantine, and even more, is  polynomially representable. In 1976, Jones, Sato, Wada, and Wiens \cite{JSW} made this explicit; they produced a degree 25 polynomial in 26 variables whose positive values, as the variables vary over nonnegative integers, is the set of primes.
	
	In analogy, we consider whether the set of primes is {\it partition theoretic,} by which we mean solutions to equations in partition functions
	(for background on partitions see~\cite{Andrews}).  This is a question of Schneider, as an example of his general philosophy and research program where classical number theoretic topics (e.g. multiplicative number theory, prime number distribution, and zeta-functions) are informed by partition theory (for example, see~\cite{Schneider1, SchneiderDawseyJust} and his works on arithmetic densities~\cite{Schneider2} and partition zeta-functions~\cite{Schneider3}).
		
	We show that the primes are partition theoretic in infinitely many ways.  Our first glimpse of this phenomenon  arises from  MacMahon's~\cite{MacMahon} $q$-series: 
	\begin{align} \label{Ua} 
	\mathcal{U}_a(q)=\sum_{n\geq1}M_a(n)\,q^n:=\sum_{0< s_1<s_2<\cdots<s_a} \frac{q^{s_1+s_2+\cdots+s_a}}{(1-q^{s_1})^2(1-q^{s_2})^2\cdots(1-q^{s_a})^2}.
	\end{align}
	For $a\geq 1$, the partition function
	$M_a(n)$  sums the products of the \emph{multiplicities} of partitions of~$n$ with $a$ different part sizes. To be precise, we have
	\begin{equation}\label{M_def}
	M_a(n)=\sum_{\substack{0<s_1<s_2<\dots<s_a\\
			n=m_1 s_1+m_2s_2+\dots+m_a s_a}} m_1m_2\cdots m_a.
	\end{equation}
	We find linear equations in these functions, with polynomial coefficients, that detect the primes. 
	
	\begin{theorem}\label{FirstExamples}  The following are true:
		
		\noindent
		(1) For positive integers $n,$ we have
		$$
		(n^2-3n+2)M_1(n)-8 M_2(n)\geq 0,
		$$
		and for $n\geq 2,$ this expression vanishes if and only if $n$ is prime.
		
		\noindent
		(2) For positive integers $n,$ we have
		$$
		(3n^3 - 13n^2 + 18n - 8)M_1(n) + (12n^2 - 120n + 212)M_2(n) -960M_3(n) \geq 0,
		$$
		and for $n\geq 2,$ this expression vanishes if and only if $n$ is prime.
	\end{theorem}
	
	\begin{remark}
	The expressions in Theorem~\ref{FirstExamples} vanish trivially when $n=1.$ 
	\end{remark}
	\begin{remark} By taking suitable linear combinations (with polynomial coefficients) of the two expressions in Theorem~\ref{FirstExamples}, it is simple to produce infinitely many further examples. For example, we have
		$$
		(3n^4 - 3n^3 - 17n^2 + 27n - 10 )M_1(n) - (240n - 400)M_2(n) - 1920 M_3(n).
		$$
		\end{remark}
	
	There are more examples of such prime detecting expressions involving the $M_a(n)$  functions.
	
	\begin{theorem}\label{General_Ua}
		Suppose that $E(n)$ is in \cref{tab:1} in the Appendix. For positive integers $n,$ we have
	$E(n)\geq 0,$ and for $n\geq 2,$ we have $E(n)=0$ 
		 if and only if $n$ is prime.
	\end{theorem}
	
	Based on a computer search for further such prime detecting partition equations, we believe that this list generates all such expressions involving MacMahon's $M_a(n)$ partition functions.
	
	\begin{conjecture*} 
		Let $\vec{P}(x):=(p_1(x), p_2(x),\dots, p_a(x))\in \Z[x]^a$ be a vector of relatively prime integer polynomials.
		For integers $n\geq 2,$ suppose that
		$$
		E(n) =  p_1(n)M_1(n)+p_2(n)M_2(n)+\dots+p_a(n)M_a(n)\geq 0,
		$$
		and  vanishes precisely on the primes. Then $E(n)$ is a $\mathbb{Q}[n]$-linear combination of entries in \cref{tab:1}.
	\end{conjecture*}
	
	Despite this conjecture,
	we show that Theorems~\ref{FirstExamples} and~\ref{General_Ua} are glimpses of a rich infinite family of such expressions.
		To make this precise, we define natural generalizations of MacMahon's $M_a(n).$ For $a \geq 1$ and a vector $\vec{a} = (v_1, v_2, \dots, v_a) \in \N^a$, we define the {\it MacMahonesque partition function}\footnote{These functions were also considered in Bachmann's master's thesis (see Chapter 5 of~\cite{BachmannMA}).}
	\begin{equation}\label{M_general_def}
	M_{\vec{a}}(n):=\sum_{\substack{0<s_1<s_2<\dots<s_a\\m_1,\ldots,m_a>0\\
			n=m_1 s_1+m_2s_2+\dots+m_a s_a}} m_1^{v_1} m_2^{v_2}\cdots m_a^{v_a},
	\end{equation}
	certain sums of  degree $|\vec{a}|:=\nu_1+\cdots+\nu_a$ monomials  in the part multiplicities.
	
	\begin{remark}	We note that $M_a(n)\neq M_{(a)}(n),$ when $a>1.$  Instead, we have $M_a(n)=M_{\vec{a}}(n)$, where
		$\vec{a}=(1,\dots,1)\in \N^a.$ 
			\end{remark}

	For these generalized MacMahon functions, we establish the existence of infinitely many linear equations with integer coefficients whose set of solutions, where $n\geq 2$, is the set of primes. 
	
	\begin{theorem}\label{GeneralDetection} If $d\geq 4$, then the following are true.
	
	\noindent
	(1) There are relatively prime integers $c_{\vec{a}}$ such that for positive integers $n$ we have
		$$
		\sum_{|\vec{a}|\leq d} c_{\vec{a}}\, M_{\vec{a}}(n) \geq 0,
		$$
		and  for $n\geq 2,$ this expression vanishes if and only if $n$ is prime. 
		
		\noindent
		(2) As $d\rightarrow \infty$, there are $\gg d^2$ many linearly independent expressions satisying (1).
	\end{theorem}
	
	\begin{examples}\label{ex:thm1.3} We offer two examples of Theorem~\ref{GeneralDetection}.
		For positive integers $n$, we find that
		\begin{displaymath}
		\begin{split}
		\Psi_1(n):=63M_{(2,2)}(n)-12&M_{(3,0)}(n) 
		-39M_{(3,1)}(n)-12M_{(1,3)}(n)\\
		&+80M_{(1,1,1)}(n)-12M_{(2,0,1)}(n)+12 M_{(2,1,0)}(n) +12 M_{(3,0,0)}(n)\geq 0,
		\end{split}
		\end{displaymath}
		and
		\begin{displaymath}
		\begin{split}
		\Psi_2(n):= 14M_{(1)}&(n) - 15M_{(2)}(n) -2 M_{(3)}(n)  + 3M_{(4)}(n)+ 30M_{(2,0)}(n)-72M_{(3,0)}(n) \\
		&-36M_{(2,1)}(n) - 6M_{(4,0)}(n)  -12M_{(3,1)}(n) + 72M_{(2,1,0)}(n) + 72M_{(3,0,0)}(n) \geq 0.
		\end{split}
		\end{displaymath}
		For $n\geq 2$ these expressions vanish if and only if $n$ is prime. Finally, we note the nontrivial identity
		$\Psi_2(n)=\frac{36}{11}\Psi_1(n).$
	\end{examples}
	
	In addition to the equations covered by Theorem~\ref{GeneralDetection}, there are many prime detecting equations involving convolutions of the MacMahonesque functions. For example, one can show that
	\begin{displaymath}
	\begin{split}
	\Psi_3(n):=10M_{(1)}&(n)-17 M_{(3)}(n)+7 M_{(5)}(n)\\+&12\sum_{i+j=n} M_{(1)}(i) \lp M_{(1)}(j)-10 M_{(3)}(j) \rp +96\!\!\sum_{i+j+\ell=n}\!\! M_{(1)}(i) M_{(1)}(j) M_{(1)}(\ell) 
	\geq 0,
	\end{split}
	\end{displaymath}
	and similarly detects primes.
	However, it turns out that this expression can be reformulated much more simply thanks to the identity
		$\Psi_3(n)=\Psi_2(n).$
	
	Such simplifications always exist.
	It turns out that arbitrary convolutions of MacMahonesque functions are always linear combinations of such functions with constant coefficients. Furthermore, linear combinations of MacMahonesque functions with polynomial coefficients can be reduced similarly. 
	To be precise, we have the following result which is implicit in the proof of Theorem~\ref{GeneralDetection}.
	
	\begin{theorem}\label{Convolution} The following are true.
		
		\noindent (1) If $\vec{\alpha}\in \Z_+^a$ and $\vec{\beta}\in \Z_+^{b}$, then
		there are rational numbers $c_{\vec{a}}$ such that for every $n$ we have
		\[ \sum_{i+j=n} M_{\vec{\alpha}}(i) M_{\vec{\beta}}(j) = \sum_{|\vec{a}| \leq {\color{black}|\vec{\alpha}|+|\vec{\beta}|+a+b-1}} c_{\vec{a}} M_{\vec{a}}(n).\]  
		
		\noindent (2) If $\vec{\alpha}\in \Z_+^a,$ then there are rational numbers $c_{\vec{a}}$ such that for every  $n$ we have
		\[ n M_{\vec{\alpha}}(n) = \sum_{{\color{black}|\vec{a}|\leq |\vec{\alpha}|+a+1}} c_{\vec{a}} M_{\vec{a}}(n).\]  
	\end{theorem}
	\begin{remarks}\ \ \newline
		\noindent
		(1) By Theorem~\ref{Convolution} (1) and (2), every convolution product of MacMahonesque functions with polynomial coefficients is  a linear combination of such functions with rational coefficients. 
		
		\noindent 
		(2) As the proof will show, Theorem~\ref{Convolution} is a partition theoretic reformulation of  beautiful results of Bachmann and K\"uhn (see Theorem~1.3 and~1.7 of~\cite{BachmannKuhn}). 
		
	\end{remarks}
	
	\begin{examples}
	The identity $\Psi_3(n)=\Psi_2(n)$ offers examples of linear expressions with constant coefficients obtained from Theorem~\ref{Convolution} (1). As a simpler example of Theorem~\ref{Convolution} (1), we have
		$$\sum_{i+j=n}M_{(1)}(i)\,M_{(1)}(j)={\color{black}\frac{1}{6}M_{(3)}(n) + 2 M_{(1,1)}(n) - \frac{1}{6} M_{(1)}(n).}
		$$
		As an example of Theorem~\ref{Convolution} (2), we have
		\begin{displaymath}
		\begin{split}
		nM_{(1,1)}(n)=\color{black}
		\frac{1}{22}\Bigl(
		-21 M_{(3, 1)}(n)&
		+72M_{(2, 2)}(n) 
		-9 M_{(1, 3)}(n)
		+24M_{(3,0)}(n)\\
		&\color{black} -24 M_{(3, 0, 0)}(n)
		-24 M_{(2, 1, 0)}(n)
		+24 M_{(2, 0, 1)}(n)
		-72M_{(1, 1, 1)}(n)
		\Bigr).
		\end{split}
		\end{displaymath}
		Finally, by iterating Theorem~\ref{Convolution} (2), we can reformulate the expression in Theorem~\ref{FirstExamples} (1) as
		\begin{displaymath}
		\begin{split}
		(n^2-3n+2)M_1(n)-&8 M_2(n)={\color{black}
			\frac{3}{11}\Psi_1(n)},
		\end{split}
		\end{displaymath}
		a linear expression in MacMahonesque functions with constant coefficients.
	\end{examples}

	To obtain these results, we make use of the theory of multiple $q$-zeta values, largely as developed by Bachmann and K\"uhn~\cite{BachmannMA, BachmannKuhn, BachmannLectures}, and the theory of quasimodular forms.
	In Section~\ref{QuasimodularBackground}, we recall facts about quasimodular forms, and we establish conditions (see Theorem~\ref{criterion2}) under which their Fourier expansions detect primes, a result which is of independent interest. In Section~\ref{UaSection}, we use the fact that
	MacMahon's $q$-series $\mathcal{U}_a(q)$ arise naturally as  quasimodular forms (see~\cite{BachmannKuhn, Andrews-Rose, rose}), combined with Theorem~\ref{criterion2}, to prove Theorems~\ref{FirstExamples} and \ref{General_Ua}. 
	
	In contrast,
	the generating series of the generalized MacMahonesque partition functions
	$$
	\mathcal U_{\vec a}(q):=\sum_{n\geq 1} M_{\vec{a}}(n)q^n 
	$$
	are generally not quasimodular. However, they occur naturally
	in the theory of  multiple $q$-zeta values.  Bachmann and K\"uhn show~\cite{BachmannKuhn} that the MacMahonesque series generate an extraordinary differential algebra whose arithmetic properties encode a rich web of identities satisfied by the MacMahonesque generating functions. To make use of quasimodularity,
	for those $\vec{a}$ with odd entries, we construct quasimodular forms  $\mathcal{U}_{\vec{a}}^{\Sym}(q)$ by taking traces over the natural action of the symmetric group~$S_a$ on~$\mathcal{U}_{\vec{a}}(q).$
	In particular, we prove the following theorem (see Theorem~\ref{lem:symsum}).
	
	\begin{theorem} 
	All quasimodular forms are linear combinations of the MacMahonesque $\mathcal{U}_{\vec{a}}^{\Sym}(q)$.
	\end{theorem}
	
	Theorem~\ref{GeneralDetection} then follows by applying Theorem~\ref{criterion2}. These details are carried out in Section~\ref{GeneralizedUaSection}.
	
	\section*{Acknowledgements}
	\noindent
	The authors would like to thank the organizers of the International Conference on
	Modular Forms and $q$-Series at the University of Cologne, where this project emerged. The authors would also like to thank Tewodros Amdeberhan, George Andrews, Henrik Bachmann, Yifeng Huang, Hasan Saad, Robert Schneider, and Ajit Singh for helpful comments which improved the manuscript. The first and second authors thank the support of the European Research Council (ERC) under the European Union's Horizon 2020 research and innovation programme (grant agreement No. 101001179) and by the SFB/TRR 191 ``Symplectic Structure in Geometry, Algebra and Dynamics'', funded by the DFG (Projektnummer 281071066 TRR 191). The third author thanks the Thomas Jefferson Fund and grants from the NSF
	(DMS-2002265 and DMS-2055118).

	\section{Quasimodular forms that detect primes}\label{QuasimodularBackground}
	Central to all of our results is a theorem (see Theorem~\ref{criterion2}) that characterizes quasimodular forms  in the space consisting of linear combinations of derivatives of Eisenstein series whose coefficients are zero at primes and positive at composite $n\geq 2$.
	
	\subsection{Preliminaries about quasimodular forms}
	To state and then prove this result, we require several important preliminary facts.
	For $\tau$ in the upper-half of the complex plane, set $q:=e^{2\pi i \tau}$. For integers $k \geq 1,$ we consider the {\it Bernoulli numbers} $B_k$ and the {\it Eisenstein series}
	\begin{equation}\label{Fourier}
	G_k(\tau) := - \dfrac{B_k}{2k} + \sum_{n \geq 1} \sigma_{k-1}(n) q^n,
	\end{equation}
	where   $ \sigma_{k-1}(n) := \sum_{d|n} d^{k-1}.$
	It is well-known (for example, see~\cite{Ono}) that $G_k$ is quasimodular if and only if $k$ is even, and modular if $k\geq 4$ and even. Furthermore, the algebra of quasimodular forms, which we denote by $\widetilde{M}$, is generated by $G_2, G_4$, and $G_6$. We also denote by $\widetilde{M}_k$ the subspace of $\widetilde{M}$ of forms with weight $k$, and $\widetilde{M}_{\leq k}$ the space of all quasimodular forms of mixed weight $\leq k$.
	
	We require the important $q$-differential operator
	\begin{align}
	D := \frac{1}{2\pi i} \dfrac{d}{d\tau} = q \dfrac{d}{dq}.
	\end{align}
	For $k \geq 0$, we observe that $D^k : q^n \mapsto n^k q^n$.
	Ramanujan famously proved the identities (for example, see \cite[Section 2.3]{Ono})\footnote{These identities are often stated for the normalized Eisenstein series $E_k := 1 + \dots$.}
	\begin{align*}
	DG_2 = - 2G_2^2 + \dfrac{5}{6} G_4, \ \ \ D G_4 = -8 G_2 G_4 + \dfrac{7}{10} G_6, \ \ \ DG_6 = -12 G_2 G_6 + \dfrac{400}{7} G_4^2.
	\end{align*}
	These identities show that $D$ defines a map on the algebra of quasimodular forms, which increases weights by $2$. Namely, we have that
	$D:\widetilde{M}_{k}\to \widetilde{M}_{k+2}.$
	
	\subsection{Some prime-detecting quasimodular forms}
	Using elementary properties of~$D,$ it is straightforward to produce quasimodular forms whose Fourier coefficients detect the primes.
	
	\begin{lemma}\label{criterion1} 
		Given 
		non-negative odd integers $k,\ell$ with $\ell>k$, consider the quasimodular forms
		\begin{align*}
		f_{k,\ell} := \bigl( D^\ell + 1 \bigr) G_{k+1} - \bigl( D^k + 1 \bigr) G_{\ell+1}.
		\end{align*}
		For $n \geq 2$, the $n$th Fourier coefficient of $f_{k,\ell}$ vanishes if and only if $n$ is prime. Furthermore, all of the coefficients of $f_{k,\ell}$ are non-negative.
		{\color{red}
			
		}
	\end{lemma}
	
	\begin{remarks}\ \ \newline
		\noindent
		(1) The assumption in Lemma~\ref{criterion1} that $k$ and $\ell$ are odd is imposed to ensure $f_{k,\ell}$ are quasimodular forms. In particular, the vanishing property is true for any non-negative integers $k$ and $\ell$.
		
		\noindent
		(2) Lemma~\ref{criterion1} appears in an unpublished note by Leli\`evre (see Lemma~2 of~\cite{Lelievre}). For completeness we offer a streamlined proof here.
	\end{remarks}
	
	\begin{proof} By inspection, the coefficient of $q^1$ vanishes. Therefore, without loss of generality we may assume that
		 $n \geq 2$ is an integer. Thanks to (\ref{Fourier}), the $n$th coefficient of $f_{k,\ell}$ is given by
		\begin{align*}
		\sum_{\substack{1 < d < n \\ d|n}} \left[ \lp 1 + n^\ell \rp d^k - \lp 1 + n^k \rp d^\ell \right].
		\end{align*}
		This sum is empty and therefore vanishing if $n$ is prime, so we assume now that $n$ is composite.   Since $d \leq \frac n2$ must hold,
		we then have
		\begin{align*}
		\lp 1 + n^\ell \rp d^k - \lp 1 + n^k \rp d^\ell &\geq d^k\lp 1 + n^\ell - \lp\dfrac{n}{2}\rp^{\ell-k} - n^k\lp \dfrac{n}{2} \rp^{\ell-k} \rp \\ &\geq d^k\lp 1 + n^\ell \lp 1 - \dfrac{1}{2^{\ell - k - 1}} \rp \rp \geq 1.
		\end{align*}
		Thus, the coefficients of $f_{k,\ell}$ are all non-negative.
	\end{proof}
	
	We now turn to the problem of determining those quasimodular forms whose Fourier expansions detect primes. To this end, we let $\mathcal{E}$ denote the \emph{quasimodular Eisenstein space}, the subspace of quasimodular forms generated additively by the even weight Eisenstein series and their derivatives. We then let $\Omega$ denote the {\it prime-detecting quasimodular forms}, where a quasimodular form
	$$
	f =\sum_{n\geq 0} b_n(f)q^n
	$$
	is  prime-detecting if for positive $n$ we have $b_n(f)\geq 0$, and for $n\geq 2$ vanish  if and only if $n$ is prime.
	
	Based on numerical evidence, we expect that the property of vanishing of the Fourier coefficients at primes is not satisfied by quasimodular forms which are not in the quasimodular Eisenstein space.
	\begin{conjecture} \label{conj:criterion}
		We have
		$\Omega = \mathcal{E} \cap \Omega.$
	\end{conjecture}
	
	\subsection{Characterizing $\mathcal{E}\cap \Omega$}
	We  give an explicit description of $\mathcal{E}\cap \Omega$ (see Theorem~\ref{criterion2}). To this end,
	for even $k\geq 6$, we define the distinguished quasimodular forms of weight $k$
	\begin{align} \label{H definition}
	H_k := \sum_{n\geq 0} b_{n}(H_k) q^n :=
	\begin{cases}
	\frac{1}{6}\left((D^2-D+1)G_2-G_4\right) &\ \ \ {\text {\rm if}}\ k=6, \\
	\frac{1}{24}(-D^2 G_{k-6} + (D^2+1) G_{k-4} - G_{k-2}) &\ \ \ {\text {\rm if}}\ k\geq 8.
	\end{cases}
	\end{align}
	The quasimodular forms $H_k$ give rise to all quasimodular forms which are linear combinations of derivatives of Eisenstein series that are prime-detecting.
	
	\begin{theorem} \label{criterion2}
		For all $n \geq 0$ and $k\geq 6$, we have $D^n H_k\in \Omega$. Conversely, if $f\in \mathcal{E}\cap \Omega$, then $f$ is a linear combination of the forms  $D^n H_k$ for $n\geq 0$ and $k\geq 6$.
	\end{theorem}
	
	Before we prove this theorem, 
	we consider two additional results on properties of these quasimodular forms $H_k$. 
	
	\begin{lemma}
		For $k\geq 6$ even, the following are true:\\
		\noindent
		\noindent
		(1) We have that $b_1(H_k)=0.$\\
		(2) The Fourier coefficients $b_{n}(H_k)$ for $n\geq 2$ are integers.\\
		\noindent
		(3) The  Fourier coefficients $b_{n}(H_k),$ where $n\geq 2,$ generate $\Z$ as a ring. 
	\end{lemma}
	
	\begin{proof} Claim (1) is clear by inspection. Therefore, we now assume that $n \geq 2$ is an integer. Noting that $b_4(H_6) = 3$ and $b_6(H_6) = 20$ are relatively prime, we see that~(3) holds for~$H_6$ assuming that~(2) holds. To show~(2) for~$H_6$, we must prove that
	$$(n^2-n+1) \sigma_1(n) - \sigma_3(n)\equiv 0\pmod{6}.$$
	  Evenness is trivial, and so we consider this expression modulo~3. If $n \not \equiv 2 \pmod{3}$, then after multiplying by $\lp n+1 \rp$, the 3-divisibility follows from 
		\begin{align*}
		\lp n^3 + 1 \rp \sigma_1(n) - \lp n+1 \rp \sigma_3(n) \equiv 0 \pmod{3},
		\end{align*}
		which is an easy consequence of Fermat's little theorem. If $n \equiv 2 \pmod{3}$, then there is a prime $p \equiv 2 \pmod{3}$ dividing $n$ with odd multiplicity $2m+1$, and then 3-divisibility follows from multiplicativity of divisor sums and $\sigma_{2j+1}\lp p^{2m+1} \rp \equiv 0 \pmod{3}$. Thus, both (2) and (3) hold for~$H_6$.
		
		We now consider even $k\geq 8.$ To prove (2), we must prove that
		\begin{align*}
		- n^2 \sigma_{k-7}(n) + \lp n^2 + 1 \rp \sigma_{k-5}(n) - \sigma_{k-3}(n) \equiv 0 \pmod{24}.
		\end{align*}
		It is straightforward to see this congruence modulo 3 from the fact that $\sigma_{j+2}(n) \equiv \sigma_j(n) \pmod{3}$ for any $j$, and so we consider the congruence above modulo 8. For any odd prime $p$, we have from $p^2 \equiv 1 \pmod{8}$ that $\sigma_{2j+1}\lp p^m \rp \equiv \sigma_1\lp p^m \rp \pmod{8}$ for any integer $j \geq 0$. Now, writing $n = 2^a m$ for $m$ odd, it follows from the multiplicativity of $\sigma_j$ that 
		\begin{align} \label{eq1}
		- &n^2 \sigma_{k-7}(n) + \lp n^2 + 1 \rp \sigma_{k-5}(n) - \sigma_{k-3}(n) \\ &\ \ \equiv \lp - 4^a m^2 \sigma_{k-7}\lp 2^a \rp + \lp 4^a m^2 + 1 \rp \sigma_{k-5}\lp 2^a \rp - \sigma_{k-3}\lp 2^a \rp \rp \sigma_1(m) \pmod{8} \notag \\ &\ \ \equiv 4^a m^2 \lp \sigma_{k-5}\lp 2^a \rp - \sigma_{k-7}\lp 2^a \rp \rp \sigma_1(m) - \lp \sigma_{k-3}\lp 2^a \rp - \sigma_{k-5}\lp 2^a \rp \rp \sigma_1(m) \pmod{8}. \notag
		\end{align}
		Now, for $j \geq 1$ we observe that
		\begin{align*}
		\sigma_{2j+1}\lp 2^a \rp - \sigma_{2j-1}\lp 2^a \rp = \sum_{b=0}^a \lp 2^{(2j+1)b} - 2^{(2j-1)b} \rp = \sum_{b=0}^a 2^{(2j-1)b}\lp 4^b - 1 \rp.
		\end{align*}
		It is therefore clear that if $j \geq 2$, $\sigma_{2j+1}\lp 2^a \rp - \sigma_{2j-1}\lp 2^a \rp \equiv 0 \pmod{8}$. If then $j=1$, from the fact that $\sigma_3\lp 2^a \rp - \sigma_1\lp 2^a \rp$ is even, we obtain from \eqref{eq1} that
		\begin{align*}
		- n^2 \sigma_{k-7}(n) &+ \lp n^2 + 1 \rp \sigma_{k-5}(n) - \sigma_{k-3}(n) \equiv - \lp \sigma_{k-3}\lp 2^a \rp - \sigma_{k-5}\lp 2^a \rp \rp \sigma_1(m) \pmod{8}.
		\end{align*}
		But then for $2j+1 = k-3$, the condition $k \geq 8$ corresponds to the condition $j \geq 2$, and so we obtain (2) for all $H_k$.
		
		We now prove (3). Let $g$ be the largest integer dividing $$24b_n(H_k) = -n^2 \sigma_{k-7}(n) + \lp n^2 + 1 \rp \sigma_{k-5}(n) - \sigma_{k-3}(n)$$ for all $n \geq 2$. We know $24 \mid g$, and therefore it will be enough to show that each of $16$, $9$, and $p$ for primes $p \geq 5$ do not divide $g$. Beginning with any $p \geq 3$ prime and $m \geq 1$ coprime to $p$, it is easy to see that
		\begin{align*}
		24 b_{pm}(H_k) \equiv \sigma_{k-5}(pm) - \sigma_{k-3}(pm) \equiv \sigma_{k-5}(m) - \sigma_{k-3}(m) \pmod{p}.
		\end{align*}
		We then observe that $24 b_{2p}(H_k) \not \equiv 0 \pmod{p}$ for $p \geq 5$ (which then entails $p$ does not divide $g$) and $24b_{5p}(H_k) \not \equiv 0 \pmod{9}$ (which then entails that 3 exactly divides $g$). Finally, for $m$ odd we see that
		\begin{align*}
		24 b_{4m}(H_k) \equiv \sigma_{k-5}(4) \sigma_{k-5}(m) - \sigma_{k-3}(4) \sigma_{k-3}(m) \pmod{16}.
		\end{align*}
		It is then easy to see that this expression is not a multiple of 16 for $m = 1$ in the case $k = 8$ and for $m = 3$ if $k \geq 10$. Therefore, we have shown that $g \mid 24$, so that $g = 24$ as required.
	\end{proof}
	
	To prove Theorem~\ref{criterion2}, we also require the following about the linear independence of the forms $D^n H_{\ell}$ and the spaces they generate.
	
	\begin{lemma}\label{lem:lindepH}\mbox{}The following are true.\\
		\noindent (1) The quasimodular forms $D^n H_\ell$ (with $n\geq 0$ and $\ell\geq 6$) are linearly independent. \\
		\noindent (2) The subspace of elements of the $D^n H_\ell$ of weight $\leq k$ is of dimension $\frac{1}{8}(k-2)(k-4)$.
	\end{lemma}
	\begin{proof}
	We first prove (1). We show the stronger statement that the quasimodular forms $D^n E_\ell$, generating $\mathcal{E}$, are linearly independent. Observe that any non-trivial relations between the quasimodular forms $D^n H_\ell$ would imply a relation between these derivatives of Eisenstein series. 
		
		To prove the claim, first of all, observe that one can restrict to quasimodular forms of homogeneous weight. Namely, by the quasimodular transformation behaviour one recovers the weight. For quasimodular forms of homogeneous weight, the statement follows from \cite[Proposition~20~(iii)]{123}
		\[
		\widetilde{M}_k = \bigoplus_{r=0}^{k/2-1} D^r M_{k-2r} \oplus \mathbb{C} D^{k/2-1} E_2.
		\]
		
		For (2), recall that the operator $D$ increases the weight by $2$ and $H_\ell$ is of weight $\ell$. We count the number of linear independent generators of weight $\leq k$, as follows:
			\begin{align*}
		\sum_{\ell \geq 3}\,\sum_{2\ell+2n\leq k} \! 1 \,=\, \frac{1}{8}(k-2)(k-4). & \qedhere
		\end{align*}
		\end{proof}
	
	\begin{proof}[Proof of Theorem~\ref{criterion2}]
		If $f$ is a quasimodular form, then $f \in \Omega$ if and only if $Df \in \Omega.$ Therefore, it suffices to check that $H_k \in \Omega$ for $k \geq 6$. We first observe that for $n \geq 2$, we have
		\begin{align*}
		(n+1) b_n(H_6) = \dfrac{n+1}{6} \left[ \lp n^2 - n + 1 \rp \sigma_1(n) - \sigma_3(n) \right] = \dfrac{1}{6} \left[ \lp n^3 + 1 \rp \sigma_1(n) - \lp n+1 \rp \sigma_3(n) \right].
		\end{align*}
		This gives the identity $(D+1) H_6=\frac{1}{6}f_{1,3}$, and so
		Lemma~\ref{criterion1} implies that $H_6\in \Omega$. 
		
		For $k \geq 8$, we calculate for $n \geq 2$ that
		\begin{align*}
		24b_n(H_k) 
		&= \sum_{d|n} \left[ -n^2 d^{k-7} + \lp n^2 + 1 \rp d^{k-5} - d^{k-3} \right] = \sum_{\substack{1 < d < n \\ d|n}} d^{k-7} \lp d^2-1 \rp \lp n^2 - d^2 \rp.
		\end{align*}
		By inspection, we see that $H_k \in \Omega$.
		
		For the converse, let $f\in \mathcal{E}\cap \Omega$ be of weight $\leq k$. As $D$ increases the weight by $2$, we can write 
		\[
		f\,=\, \sum_{\ell= 1}^{k/2} p_{2\ell}(D)\, G_{2\ell} \]
		with $p_{2\ell}$ a polynomial of degree at most $\frac{1}{2}k-\ell$. The vanishing of the Fourier coefficients at all primes is equivalent to
		\begin{equation}\label{eq:constr}
		\sum_{\ell= 1}^{k/2} p_{2\ell}(x) (x^{2\ell-1}+1) = 0.
		\end{equation}
		(Initially, this equality only holds for primes~$x$, and then, since a polynomial of bounded degree is uniquely determined by its values at primes, this is an equality of polynomials in the variable~$x$.) Observe the dimension of the space of sequences of polynomials $p_2, p_4,\ldots, p_{k}$ with the given degree bound equals
		\[
		\sum_{\ell=1}^{k/2} \left(\frac{1}{2}k-\ell+1 \right) \,=\, \frac{1}{8}k(k+2).
		\]
		Note that the left-hand side of \eqref{eq:constr} is of degree $\leq k-1$. Hence, condition~\eqref{eq:constr} gives $\leq k$ linearly independent constraints between the coefficients of the polynomials. For notational convenience, we let
		\[p_{2\ell}(x) = \sum_{j=0}^{k/2-\ell} a_{j,2\ell}\, x^j\]
		for real coefficients $a_{j,2\ell}$, and by $M$ denote the matrix of linear relations of the $a_{j,2\ell}$ given by~\eqref{eq:constr}, where the rows are parameterized by the degrees $0,\ldots,k-1$ and the columns by pairs~$(j,2\ell)$, ordered lexicographically.
		
		We show that, after removing the $0$th row of $M$, the matrix is (up to permuting the columns) in row echelon form. Namely, observe that in the $i$th row the smallest pair $(j,2\ell)$ for which the corresponding entry in $M$ is non-zero is given by
		\[
		\begin{cases}
		(0,2) & i=0, \\
		(i+1,0) & i>0 \text{ odd},\\
		(i,1) & i>0 \text{ even}.
		\end{cases}
		\]
		Hence, the matrix obtained from $M$ by removing the $0$th row is of full rank. We conclude the rank of~$M$ is at least $k-1$. As \eqref{eq:constr} is divisible by $x+1$ one can easily infer that the rank is exactly~${k-1}$. Therefore, the dimension of the space of polynomials satisfying \eqref{eq:constr} is given by
		\[
		\sum_{\ell=1}^{k/2} \left(\frac{1}{2}k-\ell+1 \right) - (k-1) = \frac{1}{8}(k-2)(k-4).
		\]
		Accordingly, this is also the dimension of the vector space of elements $D^n H_{2\ell}$ of weight $\leq k$ by \cref{lem:lindepH}(2).
		\end{proof}
			
	\section{Quasimodularity of $\mathcal{U}_a(q)$ and the proofs
		Theorems~\ref{FirstExamples} and \ref{General_Ua}}\label{UaSection}
	Here we use \cref{criterion2} to
	prove Theorems~\ref{FirstExamples} and~\ref{General_Ua}.
	
	\subsection{Proof of \cref{FirstExamples}}
	
	MacMahon's series $\mathcal{U}_a(q)$ have been the focus of much  study (for example, see \cite{rose,AOS,AAT, Bachmann, OnoSingh}). The first glimpses of their beautiful properties are captured by  the identities (for example, see~\cite{MacMahon, AOS})
	$$
	\mathcal{U}_1(q)=\sum_{n\geq1}\sigma_1(n)q^n \ \ \ \ {\text {\rm and}}\ \ \ \ 
	\mathcal{U}_2(q)=\frac{1}{8}\sum_{n\geq1} \left((-2n+1)\sigma_1(n)+\sigma_3(n)\right)q^n.
	$$
	To prove~(1), we note that the sequence we wish to study is the coefficients of the series~$F$ given by
	\begin{align*}
	F := \lp D^2 - 3D + 2 \rp \mathcal{U}_1(q) - 8 \ \mathcal{U}_2(q) = \sum_{n \geq 1} \left[ \lp n^2 - n + 1 \rp \sigma_1(n) - \sigma_3(n) \right] q^n.
	\end{align*}
	Noting the identities $F = 6 H_6$ and $\lp D+1 \rp F = f_{1,3}$, part~(1) of \cref{FirstExamples} then follows by either \cref{criterion2} or \cref{criterion1}, respectively.
	
	To prove (2), we begin by noting that \cite[p.~102]{MacMahon}
	\begin{align*}
	\mathcal{U}_3(q) = \dfrac{1}{1920}\sum_{n \geq 1} \left[ \lp 40n^2 - 100n + 37 \rp \sigma_1(n)- 10(3n-5)\sigma_3(n)+3 \sigma_5(n)  \right] q^n.
	\end{align*}
	The sequence we wish to study form the coefficients of the series $G$ given by
	\begin{align*}
	G &:= \lp 3D^3 - 13D^2 + 18D - 8 \rp \mathcal{U}_1(q) + \lp 12D^2 - 120D + 212 \rp \mathcal{U}_2(q) - 960 \ \mathcal{U}_3(q).
	\end{align*}
	It is easy to see using previous formulas for $\mathcal{U}_j(q)$ that $G = 36 H_8$, and so~(2) follows by \cref{criterion2}.

	\subsection{Proof of Theorem~\ref{General_Ua}}
	
	We recall the result of Andrews and Rose \cite[Corollary 4]{Andrews-Rose} that $\mathcal{U}_a(q)\in \widetilde{M}_{\leq 2a}$.  By \cref{criterion2}, any linear combination of derivatives of various $\mathcal U_a(q)$ detects primes if it can be represented by linear combinations of derivatives of the quasimodular forms $H_k$. Since the spaces of quasimodular forms of (mixed) weight $\leq 2a$ is finite-dimensional, the examples given in the Appendix can be verified by a straightforward linear algebra calculation\footnote{In particular, this can be done using a row-echelon basis for the spaces $\widetilde{M}_{\leq a}$ with rational coefficients.}. 
	In particular, the examples in \cref{tab:1} are exactly $6H_6\mspace{1mu}, 36 H_8\mspace{1mu}, 90 H_{10}\mspace{1mu}, 90 H_{12}\mspace{1mu}$, and $30\lp D^2 + 1 \rp H_{14} + 30 H_{16}$ respectively.

	\section{MacMahonesque series and the proof of \cref{GeneralDetection} and \cref{Convolution}}\label{GeneralizedUaSection}
	\subsection{Certain $q$-analogues of multiple zeta values}
	The proofs in the previous section rely on the fact that the $\mathcal U_a$ are quasimodular forms. Similarly, for odd $a$, we have that $\mathcal{U}_{(a)}(q)$ is quasimodular, as we have
	\begin{equation}\label{eq:Ua} \mathcal U_{(a)}(q) = \sum_{n\geq 1} M_{(a)}(n) q^n =  \frac{B_{a+1}}{2(a+1)} + G_{a+1}(\tau).
	\end{equation}
	The proof of \cref{GeneralDetection} relies on a refinement of this phenomenon. Although the generating functions
	\begin{equation}
	\mathcal{U}_{\vec{a}}(q):=\sum_{n=1}^{\infty} M_{\vec{a}}(n)q^n
	\end{equation}
	are not always a quasimodular forms, it turns out that they arise naturally in the study of $q$-analogues of multiple zeta values (see \cite[Section~6]{BachmannKuhn}), which often turn out to be related to quasimodular forms.
	
	For completeness, before we proceed with the tools for the proofs of Theorem~\ref{GeneralDetection} and Theorem~\ref{Convolution}, we consider the $q$-series generating functions for the MacMahonesque $M_{\vec{a}}(n)$ that generalize (\ref{Ua}).
	
	\begin{lemma}[{\cite[Lemma~2.5]{BachmannKuhn}}]\label{GeneratingFunctions}
		If $\vec{a}\in \N^a,$ then we have that
		$$
		\mathcal{U}_{\vec{a}}(q)=\sum_{n=1}^{\infty} M_{\vec{a}}(n)q^n = \sum_{0<s_1<s_2<\cdots<s_a} \frac{q^{s_1 + s_2 + \dots + s_a} P_{v_1}(q^{s_1}) P_{v_2}(q^{s_2}) \cdots P_{v_a}(q^{s_a})}{(1-q^{s_1})^{v_1+1} (1-q^{s_2})^{v_2+1} \cdots (1-q^{s_a})^{v_a+1}},
		$$
		where $P_k(x)$ is the $k$th Eulerian polynomial,  which are defined by
			\begin{align*}
		\dfrac{x P_n(x)}{\lp 1 - x \rp^{n+1}} := \sum_{k \geq 1} k^n x^k.		\end{align*}
	\end{lemma}
	
\begin{remark} 
In analogy with (\ref{Ua}), one might think that it makes sense to make use of the more aesthetically pleasing $q$-series
$$
\mathcal{V}_{\vec{a}}(q)=\sum_{n\geq 1} N_{\vec{a}}(n)q^n :=\sum_{0< s_1<s_2<\cdots<s_a} \frac{q^{s_1+s_2+\cdots+s_a}}{(1-q^{s_1})^{v_1+1}(1-q^{s_2})^{v_2+1}\cdots(1-q^{s_a})^{v_a+1}}.
$$
However, these $q$-series are the generating functions for the  less pleasing partition functions
$$
N_{\vec{a}}(n)=\sum_{\substack{0<s_1<s_2<\dots<s_a\\
			n=m_1 s_1+m_2s_2+\dots+m_a s_a}} \binom{m_1+v_1-1}{v_1}\binom{m_2+v_2-1}{v_2}\cdots
			\binom{m_a+v_a-1}{v_a}.
$$
To derive this expression, one notes for non-negative integers $v$ (and $|q|<1$) that
$$
\frac{q}{(1-q)^{v+1}}=\sum_{m=1}^{\infty} \binom{m+v-1}{v}q^m.  
$$
All of our results regarding $M_{\vec{a}}(n)$ could be reformulated in terms of $N_{\vec{a}}(n)$.
Finally, we note that the natural definition of $M_{\vec{a}}(n)$ in (\ref{M_general_def}) necessitates the Eulerian polynomials in Lemma~\ref{GeneratingFunctions}. For convenience, we offer the first few polynomials
$$
P_0(x)=1, \ \ P_1(x)=1,\ \ P_2(x)=x+1, \ \ P_3(x)=x^2+4x+1,\ \ P_4(x)=x^3+11x^2+11x+1,\dots.
$$
For $\vec{a}=(1,1,\dots,1),$ we recover $\mathcal{U}_{\vec{a}}(q)=\mathcal{V}_{\vec{a}}(q)$ as $P_1(x)=1.$
\end{remark}

	Write $\mathcal{Z}_q$ for the vector space generated additively by the MacMahonesque series $\mathcal{U}_{\vec{a}}$. It turns out that $\mathcal{Z}_q$ is strictly larger than the space of quasimodular forms. Furthermore, this space enjoys two important properties; namely, (i) it is an algebra, and (ii) it is closed under differentiation by $D=q\frac{d}{dq}$. Using the theory of quasi-shuffle algebras, Bachmann and Kühn proved the following theorem \cite[Theorem~1.3 and~1.7]{BachmannKuhn}. 
	\begin{theorem}\label{thm:BachmannKuhn}
		The space $\mathcal{Z}_q$ is a differential subalgebra of $\mathbb{Q}[\![q]\!]$ containing $\widetilde{M}$. 
	\end{theorem}
	Observe that this theorem directly implies that any product of (in fact, any polynomial in) the MacMahon series is a linear combination of MacMahonesque series $\mathcal U_{\vec{a}}$. Similar, any derivative of a MacMahon series is a linear combination of the MacMahonesque series $\mathcal U_{\vec{a}}$. Expressing these observations in terms of the MacMahon functions, we obtain \Cref{Convolution}.
	
	The proof of this theorem allows one to compute the product of two MacMahoneque series recursively, as we explain now. Let $A = \{z_1, z_2,\ldots \}$ be the set of \emph{letters} $z_j$, one for each positive integer.  Furthermore, we let $\Q A$ be the $\Q$-vector space generated by these letters. A \emph{word} is a non-commutative product of letters, and we write $\Q\langle A\rangle$ for the
	(non-commutative) polynomial algebra over $\Q$ generated by words with letters in $A$. On $\Q A$, we define a commutative and associative product $\diamond$ by\footnote{
		Since our $U_{\vec{a}}$ correspond to $\left(\prod_{i} v_i!\right)[\nu_1+1,\ldots,\nu_a+1]$ in~\cite{BachmannKuhn}, we chose a different normalization here.}
	\begin{equation}
	z_i \diamond z_j := \sum_{m=0}^{i+j} c_{i,j,m}\, z_{m+1}\,, 
	\end{equation}
	where we let
	\[
	c_{i,j,{i+j}} := 
	\frac{i!j!}{(i+j+1)!},\]
	and for $m<i+j$ we let
	\[
	c_{i,j,m}:=\left[(-1)^{i} \binom{i}{m+1} +(-1)^{j} \binom{j}{m+1}\right]\frac{B_{i+j-m}}{i+j-m}.
	\]
	Then, for $x, y \in A$ and $w, v \in \Q\langle A\rangle$, we recursively define the \emph{quasi-shuffle product} by $1 * w = w * 1 = w$ and
	\[xw * yv := x(w * yv) + y(xw * v) + (x \diamond y)(w * v).\]
	Identify $\Q\langle A\rangle$ with the vector space generated by finite sequences of positive integers, i.e. $z_{v_1}\cdots z_{v_a}$ corresponds to $\vec{a}=(v_1,\ldots,v_a)$. Under this identification, the quasi-shuffle product yields, e.g.,
	\[
	(1) * (1,1) = 3(1,1,1) + \frac{1}{6}(3,1) + \frac{1}{6}(1,3) -\frac{1}{3} (1,1).
	\]
	We stress that this is a formal sum of vectors, which cannot be simplified by coordinate-wise addition.
	Moreover, we extend the generalized MacMahon series linearly, i.e., for a formal finite linear combinations of finite sequences $\sum_{\vec{a}} c_{\vec{a}}\, \vec{a}$ 
	we set 
	\begin{equation}
	\mathcal{U}_{\sum_{\vec{a}} c_{\vec{a}}\, \vec{a}} := \sum\nolimits_{\vec{a}} c_{\vec{a}} \, \mathcal{U}_{\vec{a}} .
	\end{equation}
	Then by~\cite[Proposition~2.10]{BachmannKuhn}, the power series multiplication of the generalized MacMahon series can be computed recursively using the quasi-shuffle product defined above. 
	\begin{proposition}\label{ProductProperty} For all $\vec{a}\in \N^a$ and $\vec{b} \in \N^b$, one has
		\[
		\mathcal{U}_{\vec{a}}(q)\, \mathcal{U}_{\vec{b}}(q) = \mathcal{U}_{{\vec{a}}*{\vec{b}}}(q).
		\]
		Moreover, assigning weight $|\vec{a}|+\ell(\vec{a})$ to $\vec{a}$, the algebra $\mathcal{Z}_q$ becomes a filtered algebra.
	\end{proposition}
	\begin{example} As an example of Proposition~\ref{ProductProperty}, we have
		\[
		\mathcal{U}_{(1)}(q)\, \mathcal{U}_{(1,1)}(q) = 3 \mathcal{U}_{(1,1,1)} (q)+ \frac{1}{6}\mathcal{U}_{(3,1)}(q) +\frac{1}{6} \mathcal{U}_{(1,3)}(q) - \frac{1}{3} \mathcal{U}_{(1,1)}(q) =\mathcal{U}_{(1) * (1,1)}(q),
		\]
		which is equivalent to
		\[
		\mathcal{U}_1(q)\, \mathcal{U}_2(q) = 3 \mathcal{U}_3(q) + \frac{1}{6}\mathcal{U}_{(3,1)}(q) + \frac{1}{6}\mathcal{U}_{(1,3)}(q) -\frac{1}{3} \mathcal{U}_2(q).
		\]
		This example illustrates the fact that the vector space generated by the MacMahon series is not closed under multiplication, but $\mathcal{Z}_q$ is. 
	\end{example}
	
	\subsection{Symmetric group action} As we have emphasized above, a generic MacMahonesque $q$-series $\mathcal{U}_{\vec{a}}(q)$ is not quasimodular. However, their ``symmetrized sums'' are quasimodular, and provide canonical $q$-series that additively generate the ring of quasimodular forms. 
	
	To make this precise, we recall that
	 the generalized MacMahon series $\mathcal U_{\vec a}(q)$ admit the natural action of the $a$-th symmetric group $S_a$ by
	\begin{align}
	\sigma \mathcal U_{\vec a}(q) := \mathcal U_{\sigma \vec a}(q), 
	\end{align}
	where for $\sigma\in S_a$ we let $\sigma \vec a := \lp v_{\sigma(1)}, \dots, v_{\sigma(a)} \rp$.
	For those $\vec{a}$ consisting of only odd entries, we consider the {\it symmetrized MacMahonesque} $q$-series
	 \begin{equation}
	\mathcal{U}_{\vec{a}}^{\Sym}(q):= \sum_{\sigma\in S_a} \sigma\mathcal U_{\vec a}(q).
	 \end{equation}

	\begin{theorem}\label{lem:symsum} The following are true.
	
	\noindent
	(1) For all $\vec{a}\in (2\Z_+-1)^a$, we have that $\mathcal{U}_{\vec{a}}^{\Sym}(q)\in \widetilde{M}_{\leq |\vec{a}|+a}.$
	
	\noindent
	(2) 
	The $\mathcal{U}_{\vec{a}}^{\Sym}(q)$ with $|\vec{a}|+a\leq k$ additively generate $\widetilde{M}_{\leq k}$. 
	
	\noindent
	(3) The ring of quasimodular forms is additively generated as follows
		$$
\widetilde{M}= \left
\langle  1,\hskip.02in \mathcal{U}_{(v_1)}, \hskip.02in \mathcal{U}^{\Sym}_{(v_1, v_2)}, \hskip.02in \mathcal{U}^{\Sym}_{(v_1,1,\dots,1)}, \hskip.02in \mathcal{U}^{\Sym}_{(v_1,v_2,1,\dots,1)} \ |\ v_1, v_2 \ {\text {\rm positive odds}}\right \rangle.
$$
	\end{theorem}
	\begin{remark} Theorem~\ref{lem:symsum} (1) is given as an exercise in~\cite{BachmannLectures}.  Parts (2) and (3) of the theorem are the new observations.
	\end{remark}
	
	\begin{proof} We now prove (1).
		Symmetric sums in the context of quasi-shuffle algebras have been studied broadly. In particular, by \cite[Theorem~13]{Hoffman}
		we have
		\begin{equation}\label{Usymformula}
	\mathcal{U}_{\vec{a}}^{\Sym}(q)=\sum_{\sigma\in S_a} \sigma\mathcal U_{\vec a}(q) = \sum_{B\in \Pi_a} c(B) 
		\prod_{\beta\in B} \mathcal{U}_{{\color{black}\diamond_{\beta}}}(q),
		\end{equation}
		where $\Pi_a$ denotes all set partitions of the set $\{1,\ldots,a\}$, ${\color{black}\diamond_{\beta}:=\diamond_{i\in \beta} (v_i)}$ denotes the $\diamond$-product of the~$z_{v_i}$ (written as a finite sequence), and
		\[
		c(B) = (-1)^{a-|B|} \prod_{\beta\in B} (|\beta|-1)!
		\]
		is a constant. 
		In particular, this expresses $\sum_{\sigma\in S_a} \sigma\mathcal U_{\vec a}(q) $ as a polynomial in the modular Eisenstein series (up to constant terms)  $\mathcal{U}_{(v)}$ for certain $v\in \Z_+$. We now show that, in fact, $v$ only takes odd values.
		
		Recall that the entries $v_1,\ldots,v_a$ of $\vec{a}$ are all odd. Now, observe that for odd $i,j$, we have that $c_{i,j,m}=0$ if $m$ is odd. Namely, for odd $i,j,m$ the Bernoulli number $B_{i+j-m}$ is zero, unless $i+j-m=1$. In the latter case $m+1>i$ and $m+1>j$, and hence $\binom{i}{m+1}=\binom{j}{m+1}=0$. Hence, in the $\diamond$-product of $z_i$ and $z_j,$ where $i$ and $j$ are odd, we find that the coefficients of $z_{m+1}$ vanish if $m$ is odd. We conclude that $\sum_{\sigma\in S_a} \sigma\mathcal U_{\vec a}(q) $ can be expressed as a polynomial in $\mathcal{U}_{(v)}$ for certain odd integers~$v$. Now, (1) follows from~\eqref{eq:Ua}. 
	
		We now turn to the proof of~(2) and (3). By \eqref{eq:Ua}, the symmetrized MacMahonesque $q$-series generate all the Eisenstein series. If $\vec{a}=(v_1,v_2)$ is of length $2$, then by brute force using~\eqref{Usymformula} we find that $\mathcal{U}_{\vec{a}}^{\Sym}(q)$ is---up to an Eisenstein series---equal to the product of the Eisenstein series $\mathcal{U}_{(v_1)}(q)$ and $\mathcal{U}_{(v_2)}(q)$. It follows that $\mathcal{U}_{\vec{a}}^{\Sym}(q)$ with $|\vec{a}|+a\leq k$ additively generate the algebra of modular forms $M_{\leq k}$, as it is well-known that the space of modular forms is generated by polynomials in Eisenstein series of degree at most $2$ (for example, see \cite[Section~5]{Zagier}). Finally, using~\eqref{Usymformula} again, observe that for $\vec{a}=(v,1,1,\ldots,1)$ and $\vec{a}=(v_1,v_2,1,1,\ldots,1)$ one obtains arbitrary powers of the quasimodular $\mathcal{U}_{1}$ times Eisenstein series and products of Eisenstein series respectively, modulo quasimodular forms of lower depth. Here, we recall every quasimodular form can be written as a polynomial in the quasimodular Eisenstein series~$\mathcal{U}_1$ (which equals $G_2$ up to a constant), and the depth is defined as the degree in $\mathcal{U}_1$. Claims~(2) and (3) now follow by induction on the depth. 
	\end{proof}

\subsection{Proof of Theorem~\ref{GeneralDetection}}
	For $d=4$, we have example~$\Psi_1(n)$ in the Introduction. Multiplying by~$n^{d-4}$ and using \cref{Convolution} implies the existence for all $d\geq 4$, thereby confirming (1).
				
				We turn to (2). By \cref{lem:symsum} (2) the subspace of all $\mathcal U_{\vec{a}}(q)$ of weight at most $k$ generates~$\widetilde{M}_k$. 
				By \cref{lem:lindepH} there are $\gg k^2$ linearly independent quasimodular forms of weight at most $k$ for which the $n$th Fourier coefficient ($n\geq 2$) is non-negative and vanish if and only if $n$ is prime. We also note that the coefficient of $q^1$ in these forms also vanish. This proves (2).

\subsection{Proof of Theorem~\ref{Convolution}}
	The first part of the statement follows from \cref{ProductProperty}. In particular, the bound on $\vec{a}$ in the sum in the statement of the theorem, follows from estimating $\ell(\vec{a})\geq 1$ and the fact that $\mathcal{Z}_q$ is a filtered algebra.	The second part of the statement follows from the fact that the operator $D$ preserves $\mathcal{Z}_q$ and increases the weight by $2$ (see \cref{thm:BachmannKuhn} and \cite[Theorem~1.7]{BachmannKuhn}).

		\begin{landscape}
			\section*{Appendix}
			
			\noindent
			\renewcommand{\arraystretch}{1.95}
			\begin{small}
				\begin{table}[h]
					\caption{Prime detecting expressions for \cref{General_Ua}. See~\eqref{H definition} for the definition of the quasimodular forms $H_k$.}
					\label{tab:1} 
					\begin{tabular}{|c|c|}
						\hline
						Prime Detecting Expression & Quasimodular form\\ \hline\hline
						$\displaystyle (n^2 - 3n + 2)M_1(n) -8M_2(n)$ & $6H_6$ \\\hline
						$\displaystyle (3n^3 - 13n^2 + 18n - 8)M_1(n) + (12n^2 - 120n + 212)M_2(n) -960M_3(n)$ & $36H_8$  \\\hline
						$\displaystyle (25n^4 - 171n^3 + 423n^2 - 447n + 170)M_1(n) + (300n^3 - 3554n^2 + 12900n - 14990)M_2(n) + $ &\\
						$\displaystyle\quad (2400n^2 - 60480n + 214080)M_3(n) -725760M_4(n)$ & $90H_{10}$  \ \\\hline
						$\displaystyle (126n^5 - 1303n^4 + 5073n^3 - 9323n^2 + 8097n - 2670)M_1(n) +$ &\\
						$\displaystyle\quad(3024n^4 - 48900n^3 + 288014n^2 - 737100n + 695490)M_2(n) +$  & \\
						$\displaystyle\quad (60480n^3 - 1510080n^2 + 10644480n - 23496480)M_3(n) +$& \\
						$\displaystyle\quad (725760n^2 - 36288000n + 218453760)M_4(n) -580608000M_5(n)$  & $90H_{12}$ \\\hline
						$\scriptstyle\quad (300n^8 - 1542n^7 - 33049n^6 + 377959n^5 - 1651959n^4 + 3726801n^3 - 4575760n^2 + 2903750n - 746500)M_1(n) +$&\\
						$\scriptstyle\quad (12000n^7 - 91008n^6 - 2799900n^5 + 50637162n^4 - 351366300n^3 + 1239098170n^2 - 2210467000n + 1585493500)M_2(n) +$&\\
						$\scriptstyle\quad (432000n^6 - 3548160n^5 - 236343840n^4 + 5133219840n^3 - 42370071840n^2 + 161101416000n - 236150560800)M_3(n) +$&\\
						$\scriptstyle\quad (12096000n^5 - 72817920n^4 - 17599680000n^3 + 396192142080n^2 - 3123876672000n + 8555162112000)M_4(n) +$& \\
						$\scriptstyle\quad(193536000n^4 - 1056513024000n^2 + 21310248960000n - 112944125952000)M_5(n) + $&\\
						$\scriptstyle\quad (-46495088640000n + 604436152320000)M_6(n) -1115882127360000M_7(n)$  & $\scriptstyle 30(D^2+1)H_{14}+30H_{16}$
						\\\hline
					\end{tabular}
				\end{table}
			\end{small}
		\end{landscape}


\begin{thebibliography}{99}
		
			\bibitem{Matiyasevich} Y. V. Matiyasevich, \emph{Enumerable sets are Diophantine}, Doklady 
			Akademii Nauk SSSR (in Russian), \textbf{191} (1970),  279-282. 
			
			\bibitem{JSW} J. P. Jones, D. Sato,  H. Wada, and D. Wiens, \emph{Diophantine representation of the set of prime numbers}, American Mathematical Monthly \textbf{83} (1976), 449-464.

\bibitem{Andrews} G. E. Andrews, 
{\it The theory of partitions}, Reprint of the 1976 original, Cambridge Mathematical Library, Cambridge University Press, Cambridge, 1998.

\bibitem{Schneider1} R. Schneider, \emph{Eulerian series, zeta functions, and the arithmetic of partitions}, Thesis (Ph.D.)- Emory University, 2018. 220 pages.

\bibitem{SchneiderDawseyJust}  M. Locus Dawsey, M. Just and R. Schneider, \emph{A ``supernormal'' partition statistic}, J. Number Theory {\textbf 241} (2022), 120-141.

\bibitem{Schneider2} R. Schneider and A. V. Sills, \emph{Analysis and combinatorics of partition zeta functions},
Int. J. Number Theory \textbf{17}, No. 3 (2021), 805-814.

\bibitem{Schneider3} R. Schneider, \emph{Partition zeta functions}, Research in Number Theory
\textbf{2} (2016), Article 9.

\bibitem{MacMahon} P. A. ~MacMahon, \emph{Divisors of Numbers and their Continuations in the Theory of Partitions}, Proc. London Math. Soc. (2) \textbf{19} (1920), no.1, 75-113 
[also in Percy Alexander MacMahon Collected Papers, Vol.2, pp. 303--341 (ed. G.E. Andrews), MIT Press, Cambridge, 1986].

\bibitem{BachmannMA} H.~Bachmann, \emph{Multiple Zeta-Werte und die Verbindung zu Modulformen durch Multiple Eisensteinreihen}, Masterarbeit in Mathematik, Universit\"at Hamburg, 2012.

\bibitem{BachmannKuhn} H.~Bachmann and U.~K\"uhn, \emph{The algebra of generating functions for multiple divisor sums and applications to multiple zeta values},. Ramanujan J. \textbf{40} (2016), no.3, 605--648.

\bibitem{BachmannLectures} H.~Bachmann, \emph{Lectures on Multiple zeta values and modular forms}, Nagoya University, Spring 2020 {\url{https://www.henrikbachmann.com/uploads/7/7/6/3/77634444/mzv_mf_2020_v_5_4.pdf}}.
	
\bibitem{Andrews-Rose} G. E.~Andrews and S. C. F.~Rose, \emph{MacMahon's sum-of-divisors functions, Chebyshev polynomials, and quasimodular forms},  J. Reine Angew. Math. \textbf{676} (2013), 97--103.

\bibitem{rose} S. C. F. ~Rose, \emph{Quasimodularity of generalized sum-of-divisors functions}, Research in Number Theory \textbf{1} (2015), Paper No. 18, 11 pp.


\bibitem{Ono} K. Ono, \emph{The web of modularity: arithmetic of the coefficient of modular forms and $q$-series}. CBMS Regional Conference Series in Mathematics, {\bf 102}, Amer. Math. Soc., Providence, RI 2004.
			
\bibitem{Lelievre} S. Leli\`evre, \emph{Quasimodular forms with Fourier coefficients zero at primes}, unpublished note.
			
							
\bibitem{123} Don Zagier, \emph{Elliptic modular forms and their applications},  In The 1-2-3 of modular forms, 
		Universitext, pages 1--103. Springer, Berlin, 2008.
		
\bibitem{AOS} T.~Amdeberhan, K. Ono and A. Singh, \emph{MacMahon's sums-of-divisors and allied $q$-series}, {\url{https://arxiv.org/abs/2311.07496}}, preprint.

\bibitem{AAT} T. ~Amdeberhan, G. E. ~Andrews, and R. ~Tauraso,
			\emph{ Extensions of MacMahon's sums of divisors}, Res. Math. Sci. {\bf 11}, 8 (2024).

\bibitem{Bachmann} H. Bachmann, {\it MacMahon's sums-of-divisors and their connection to multiple Eisenstein series}, arXiv:2312.06132

\bibitem{OnoSingh} K. Ono and A. Singh, \emph{Remarks on MacMahon's $q$-series}, {\url{https://arxiv.org/abs/2402.08783}}, preprint.
		
\bibitem{Hoffman} M. E. Hoffman, \emph{Quasi-shuffle algebras and applications}, Algebraic combinatorics, resurgence, moulds and applications (CARMA). Vol. 2, 327–348.
IRMA Lect. Math. Theor. Phys., 32, EMS Publishing House, Berlin, 2020.
			
\bibitem{Zagier}  D. Zagier, {\it Modular forms whose Fourier coefficients involve zeta-functions of quadratic fields}. In: Serre, J.-P., Zagier,
D.B. (eds.), modular functions of one variable, VI (proceedings international conference, university Of Bonn, Bonn
					
				
			
		\end{thebibliography}
	\end{document}